\documentclass[reqno]{amsart} %[twoside]
\usepackage[T1]{fontenc}
\usepackage[c5paper, text={126truemm,178truemm},centering]{geometry}

\def\thead{submitted}

\numberwithin{equation}{section}

\makeatletter
\def\@seccntformat#1{%
  \protect\textup{%
    \protect\@secnumfont
    \expandafter\protect\csname format#1\endcsname % <--- added
    \csname the#1\endcsname
    \protect\@secnumpunct
  }%
}
\makeatother
\usepackage{xpatch}%{etoolbox}
\xpatchcmd{\section}{\scshape}{\bfseries\scshape}{}{}
\xpatchcmd{\subsection}{\bfseries}{\bfseries\scshape}{}{}
\xpatchcmd{\subsubsection}{\itshape}{\scshape}{}{}
\xpatchcmd{\proof}{\itshape}{\bfseries}{}{}

\usepackage{latexsym}
\usepackage{amssymb,amsmath,amsopn,amsthm,amssymb}
\usepackage{datetime,float}
\usepackage{graphicx,color,epsfig}      %To insert ps figures
\usepackage{tikz} %for drawing figures
\usetikzlibrary{calc,arrows,decorations.markings,intersections,matrix,shadings,shapes,through} %3d,
\newcommand{\pgfextractangle}[3]{%
	\pgfmathanglebetweenpoints{\pgfpointanchor{#2}{center}}
		{\pgfpointanchor{#3}{center}}
	\global\let#1\pgfmathresult  
}

\usepackage[all]{background}
 \SetBgColor{black}%
 \SetBgOpacity{1.0}%
 \SetBgScale{1}% use \scalebox instead if href is in text
 \SetBgPosition{current page.south}%
 \SetBgAngle{0.0}%
 \SetBgVshift{10pt}%
 \SetBgContents{\scalebox{0.7}{\color{blue!70}%
	 (\monthname\ \number\day, \number\year)
	 \ \copyright\ \shortauthors\ 
	 all rights reserved%
 }}

\usepackage[font=small]{caption} %,labelfont={sc}

\usepackage[inline]{enumitem}
\setlist[enumerate,1]{label={\upshape(\arabic*)}}%[label={\upshape(\roman*)}]
\usepackage[backref=page]{hyperref}
\hypersetup{colorlinks, citecolor=[rgb]{0.0,0.3,0.1}, linkcolor=[rgb]{0.3,0.0,0.1},
						urlcolor=[rgb]{0.0,0.0,0.4}}

\makeatletter % % % % % % % % % % % % % % % % % % % % % % % % % % % % % % % %
\AtBeginDocument{%
 \hypersetup{%
  pdftitle={\@title},
  pdfauthor={\authors},
  pdfsubject={\@subjclass},
  pdfkeywords={\@keywords}
 }
 \begin{picture}(0,0)(0,0)\put(340,50){\hbox to 0mm{\hss\small\itshape\thead}}\end{picture}
 \null\vskip-12mm\null%
}%
\makeatother % % % % % % % % % % % % % % % % % % % % % % % % % % % % % % % %

\usepackage[nobysame]{amsrefs}
\makeatletter
\def\print@backrefs#1{%\def\number@ne{#1}\def\n@thing{}%
    \space\SentenceSpace\expandafter\ifx\csname br@#1\endcsname\relax%
    \relax\else{\rm~$\langle$\csname br@#1\endcsname$\rangle$}\fi%
}
\makeatother

\BibSpec{article}{%
    +{}  {\PrintAuthors}                {author}
    +{,} { \textit}                     {title}
    +{.} { }                            {part}
    +{:} { \textit}                     {subtitle}
    +{,} { \PrintContributions}         {contribution}
    +{.} { \PrintPartials}              {partial}
    +{,} { }                            {journal}
    +{}  { \textbf}                     {volume}
    +{}  { \PrintDatePV}                {date}
    +{,} { \issuetext}                  {number}
    +{,} { \eprintpages}                {pages}
    +{,} { }                            {status}
    +{,} { \url}                        {url}    % <---- ADDED
    +{,} { \PrintDOI}                   {doi}
    +{,} { available at \eprint}        {eprint}
    +{}  { \parenthesize}               {language}
    +{}  { \PrintTranslation}           {translation}
    +{;} { \PrintReprint}               {reprint}
    +{.} { }                            {note}
    +{.} {}                             {transition}
    +{}  {\SentenceSpace \PrintReviews} {review}
}

\BibSpec{book}{%
    +{}  {\PrintPrimary}                {transition}
    +{,} { \textit}                     {title}
    +{.} { }                            {part}
    +{:} { \textit}                     {subtitle}
    +{,} { \PrintEdition}               {edition}
    +{}  { \PrintEditorsB}              {editor}
    +{,} { \PrintTranslatorsC}          {translator}
    +{,} { \PrintContributions}         {contribution}
    +{,} { }                            {series}
    +{,} { \voltext}                    {volume}
    +{,} { }                            {publisher}
    +{,} { }                            {organization}
    +{,} { }                            {address}
    +{,} { \PrintDateB}                 {date}
    +{,} { }                            {status}
    +{,} { \url}                        {url}    % <---- ADDED
    +{,} { \PrintDOI}                   {doi}    % <---- ADDED
    +{,} { available at \eprint}        {eprint}    % <---- ADDED
    +{}  { \parenthesize}               {language}
    +{}  { \PrintTranslation}           {translation}
    +{;} { \PrintReprint}               {reprint}
    +{.} { }                            {note}
    +{.} {}                             {transition}
    +{}  {\SentenceSpace \PrintReviews} {review}
}

\newtheorem{thm}{Theorem}[section]
\newtheorem{lem}[thm]{Lemma}
\newtheorem{cor}[thm]{Corollary}

\newtheorem{conj}{Conjecture}
\newtheorem{ques}{Question}

\theoremstyle{definition}
\newtheorem{defn}[thm]{Definition}
\newtheorem{rem}[thm]{Remark}

\renewcommand{\Re}{\mathbb R}

\newcommand{\B}{\mathbf B}

\newcommand{\D}{\mathcal{D}}
\newcommand{\BB}{\mathcal{B}}
\renewcommand{\SS}{\mathcal{S}}
\newcommand{\K}{\mathcal{K}}
\newcommand{\C}{\mathcal{C}}

\DeclareMathOperator{\inter}{int}
\DeclareMathOperator{\bd}{\partial}

\DeclareMathOperator{\conv}{conv}

\DeclareMathOperator{\area}{area}

\DeclareMathOperator{\cl}{cl}

\title[Tiling a circular disc with congruent pieces]
      {Tiling a circular disc with congruent pieces}

\subjclass[2010]{52C20; 52C22}
\keywords{tiling, dissection, monohedral, topological disc, Jordan region}

\author[\'A. Kurusa]{\'Arp\'ad Kurusa}
\thanks{\'A. Kurusa's
  research was supported by NFSR of Hungary (NKFIH) under
  grant numbers K~116451 and KH\_18~129630,
  and by the Ministry of Human Capacities, Hungary grant 20391-3/2018/FEKUSTRAT.}
 \address{\'A. Kurusa, \upshape
  Alfr\'ed R\'enyi Institute of Mathematics, Hungarian Academy of Sciences,
  Re\'altanoda u. 13-15, H-1053 Budapest, Hungary;
  \emph{and}
  Bolyai Institute, University of Szeged,
  Aradi v\'ertan\'uk tere 1, 6725 Szeged, Hungary.
 }
 \email{kurusa@math.u-szeged.hu}
 \urladdr{http://www.math.u-szeged.hu/tagok/kurusa}%

\author[Z. L\'angi]{L\'angi Zsolt}
 \thanks{Z. L\'angi's research is supported by the NFSR of Hungary (NKFIH) under grant number K-119670, the J\'anos Bolyai Research Scholarship of the Hungarian Academy of Sciences, and grants BME FIKP-V\'IZ and the \'UNKP-19-4 New National Excellence Program by the Ministry for Innovation and Technology.}
 \address{Z. L\'angi, \upshape
  MTA-BME Morphodynamics Research Group and Dept. of Geometry,
  Budapest University of Technology, Egry J\'ozsef u. 1., 1111 Budapest, Hungary.
 }
 \email{zlangi@math.bme.hu}
 \urladdr{http://math.bme.hu/~zlangi/}%

\author[V. V\'{\i}gh]{Viktor V\'{\i}gh}
 \thanks{V. V\'{\i}gh's research was supported by NFSR of Hungary (NKFIH) under
  grant number K~116451,
  and by the Ministry of Human Capacities, Hungary grant 20391-3/2018/FEKUSTRAT.}
 \address{V. V\'{\i}gh, \upshape
  Dept. of Geometry, Bolyai Institute, University of Szeged, 
  Aradi v\'ertan\'uk tere 1, H-6720 Szeged, Hungary \emph{and} Dept. of Natural Sciences and Engineering, Faculty of Mechanical Engineering and Automation, John von Neumann University, Izs\'aki \'ut 10, H-6000 Kecskem\'et, Hungary.
 }
 \email{vigvik@math.u-szeged.hu}
 \urladdr{http://www.math.u-szeged.hu/tagok/vigvik/}%

\begin{document}

\begin{abstract}
In this note we prove that
any monohedral tiling of the closed circular unit disc
with $k \leq 3$ topological discs as tiles
has a $k$-fold rotational symmetry.
This result yields the first nontrivial estimate about the minimum number
of tiles in a monohedral tiling of the circular disc
in which not all tiles contain the center,
and the first step towards answering a question of Stein
appearing in the problem book of Croft, Falconer and Guy in 1994.
\end{abstract}

\null\vskip-16mm\null%

\maketitle

\null\vskip-16mm\null%

\section{Introduction}

A \emph{tiling} of a convex body $\mathcal K$ in Euclidean $d$-space $\Re^d$
is a finite family of compact sets in $\Re^d$ with mutually disjoint interiors,
called \emph{tiles}, whose union is $\mathcal K$.
A tiling is \emph{monohedral}, if all tiles are congruent.

In this paper we deal with the monohedral tilings of the closed circular
unit disc $\BB^2$ with center $O$, in which the tiles are Jordan regions;
i.e. are homeomorphic to a closed circular disc.
The easiest way to generate such tilings,
which we call \emph{rotationally generated tilings},
is to rotate around $O$ a simple, continuous curve
connecting $O$ to a point on the boundary $\SS^1$ of~$\BB^2$.
The following question, based on the observation
that any tile of such a monohedral tiling of $\BB^2$ contains $O$,
seems to arise regularly in recreational mathematical circles~\cite{MathOverflow}:

\begin{ques}\label{ques:existence}%\hspace{-2mm}
 Are there monohedral tilings of $\BB^2$ in which not all of the tiles contain $O$?
\end{ques}

The answer to Question~\ref{ques:existence} is affirmative;
the usual examples to show this are the first two configurations in Figure~\ref{fig:twelve}.
The following harder variant %of Question~\ref{ques:existence}
is attributed to Stein by Croft, Falconer and Guy in \cite[last paragraph on p.~87]{CroftFalconerGuy1994}.

\begin{ques}[Stein]\label{ques:Stein}
 Are there monohedral tilings of $\BB^2$ in which $O$ is
 in the interior of a tile? 
\end{ques}

A systematic investigation of monohedral tilings of $\BB^2$ was started
in \cite{HaddleyWorsley2016} by Haddley and Worsley.
In their paper they called a monohedral tiling \emph{radially generated},
if every tile is \emph{radially generated}, meaning that
its boundary is a continuous simple curve consisting of three parts:
a circular arc of length $\alpha$ and two other curves one of which
is the rotation of the other one about their common point by angle $\alpha$.
The following ambitious conjecture
%, which would clearly imply a refuting answer to Question~\ref{ques:Stein},
appears in \cite[Conjecture 6.1]{HaddleyWorsley2016}.

\begin{conj}[Haddley and Worsley]\label{conj:radiallygenerated}
 Every monohedral tiling is a subtiling of a radially generated tiling.
\end{conj}

A similar problem was investigated in~\cite{Goncharov} by Goncharov,
who, for any $O$-symmetric convex body in $\Re^d$,
determined the smallest number of congruent copies
of a subset of the body that cover the body.
In the spirit of this approach we raise the following
variant of Question~\ref{ques:existence}:

\begin{ques}\label{q:minima}
 What is the minimum cardinality $n(\BB^2)$ of a monohedral tiling of $\BB^2$
 in which not all of the tiles contain $O$?
\end{ques}

As the configurations in Figure~\ref{fig:twelve}
show, we have $n(\BB^2) \leq 12$.
On the other hand, the lower bound $n(\BB^2) \geq 3$ is also relatively easy to prove:
it was posed as a problem in 2000
on the Russian Mathematical Olympiads \cite{MoscowMOs2000-5}.
Presently, to the authors' knowledge,
the best bounds on $n(\BB^2)$ are still the trivial ones: $3 \leq n(\BB^2) \leq 12$.

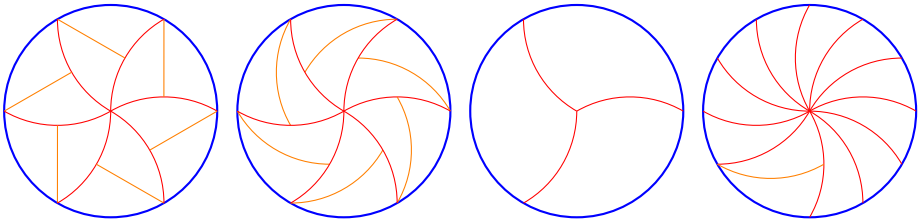
\begin{figure}[H]\centering\vskip-.5\baselineskip
 \begin{tikzpicture}[x=14mm,y=14mm]
  \coordinate (O) at (0,0);
  \def\tempr{1}
  \draw[blue,thick] (O) circle (\tempr);
  \foreach\x in {0,...,5}	{
   \draw[red, name path=curve] (60*\x:{\tempr}) arc (60*(\x+1):60*(\x+2):1);
   \path[name path=line] (60*\x+60:{\tempr}) -- (60*\x:{\tempr/2});
   \draw [orange,name intersections={of=curve and line}] (60*\x+60:{\tempr}) -- (intersection-1);
  }
 \end{tikzpicture}
 \ %quad
 \begin{tikzpicture}[x=14mm,y=14mm]
  \coordinate (O) at (0,0);
  \def\tempr{1}
  \draw[blue,thick] (O) circle (\tempr);
  \foreach\x in {0,...,5}	{
   \coordinate (P) at ({\tempr*cos(60*\x)},{\tempr*sin(60*\x)});
   \draw[orange,thin,rotate around={{60*\x+30}:(P)}] (P) arc (0:60:{\tempr});
   \draw[red,thin,rotate around={{60*\x+60}:(P)}] (P) arc (0:60:{\tempr});
  }
 \end{tikzpicture}
 \ %quad
 \begin{tikzpicture}[x=14mm,y=14mm]
  \coordinate (O) at (0,0);
  \def\tempr{1}
  \draw[blue,thick] (O) circle (\tempr);
  \foreach\x in {0,1,2}	{
   \coordinate (P) at ({\tempr*cos(120*\x)},{\tempr*sin(120*\x)});
%   \draw[orange,thin,rotate around={{60*\x+30}:(P)}] (P) arc (0:60:{\tempr});
   \draw[red,thin,rotate around={{120*\x+60}:(P)}] (P) arc (0:60:{\tempr});
  }
 \end{tikzpicture}
% \quad
% \begin{tikzpicture}[x=14mm,y=14mm]
%  \coordinate (O) at (0,0);
%  \def\tempr{1}
%  \draw[blue,thick] (O) circle (\tempr);
%  \foreach\x in {0,2,4}	{
%   \coordinate (P) at ({\tempr*cos(60*\x)},{\tempr*sin(60*\x)});
%   \draw[orange,thin,rotate around={{60*\x+30}:(P)}] (P) arc (0:60:{\tempr});
%   \draw[red,thin,rotate around={{60*\x+60}:(P)}] (P) arc (0:60:{\tempr});
%  }
%  \foreach\x in {1,3,5}	{
%   \coordinate (P) at ({\tempr*cos(60*\x)},{\tempr*sin(60*\x)});
%   \draw[red,thin,rotate around={{60*\x+30}:(O)}] (O) arc (0:-60:{\tempr});
%   \draw[red,thin,rotate around={{60*\x+60}:(P)}] (P) arc (0:60:{\tempr});
%  }
% \end{tikzpicture}
 \ %quad
 \begin{tikzpicture}[x=14mm,y=14mm]
  \coordinate (O) at (0,0);
  \def\tempr{1}
  \draw[blue,thick] (O) circle (\tempr);
  \foreach\x in {0,1,2,3,4,5,6,7,9,10,11}	{
   \coordinate (P) at ({\tempr*cos(30*\x)},{\tempr*sin(30*\x)});
%   \draw[orange,thin,rotate around={{60*\x+30}:(P)}] (P) arc (0:60:{\tempr});
   \draw[red,thin,rotate around={{30*\x+60}:(P)}] (P) arc (0:60:{\tempr});
  }
  \foreach\x in {7}	{
   \coordinate (P) at ({\tempr*cos(30*\x)},{\tempr*sin(30*\x)});
   \draw[orange,thin,rotate around={{30*\x+30}:(P)}] (P) arc (0:60:{\tempr});
%   \draw[magenta,thin,rotate around={{60*\x+60}:(P)}] (P) arc (0:60:{\tempr});
  }
 \end{tikzpicture}
% \quad
% \begin{tikzpicture}[x=14mm,y=14mm]
%  \draw[blue,thick] (0,0) circle (1);
%  \foreach\x in {0,...,5}	{
%   \draw[red,name path=curve] (60*\x:1)+(60*\x+120:0.422) arc(60*(\x+1):60*(\x+2):1) -- (0,0);
%   \path[red,name path=line] (0,0) -- (60*\x+60:1);
%   \draw[orange,name intersections={of=curve and line}] (60*\x+60:1) -- (intersection-1);
%   \draw[orange,rotate around={60:(intersection-1)}, name intersections={of=curve and line}] (60*\x+60:1) -- (intersection-1);
%  }
% \end{tikzpicture}%
% \quad
% \begin{tikzpicture}[x=12mm,y=12mm]
%  \def\s{0.376470588}
%  \def\k{2};
%  \def\prad{2.65*\s};
%  \def\qrad{0.66*\s};
%  \node (p) at (-\s,0) {};
%  \node (q) at (\s,0) {};
%  \path [name path=arcp] (p) circle (\prad);
%  \path [name path=arcq] (q) circle (\qrad);
%  \path [name intersections={of=arcp and arcq, name=i, total=\t}];
%  \node (r) at (i-1) {};
%  \pgfresetboundingbox
%  \node [draw,blue,thick] at (p) [circle through=(r)] {};
%  \pgfextractangle{\aq}{q}{r};
%  \pgfextractangle{\ap}{p}{r};
%  \foreach \x in {0,...,5}	{
%   \draw[red] (-\s,0) ++(60*\x:2*\s) ++(60*\x+\aq:\qrad) arc(60*\x+\aq:60*\x+\aq+180:\qrad/2) arc(60*\x+\aq+240:60*\x+\aq+60:\qrad/2) arc (60*\x+60+\ap:60*\x+120+\ap:\prad) arc (60*\x+\aq+120:180+60*\x+\aq+120:\qrad/2);
%   \draw[orange] (-\s,0) ++(60*\x:2*\s) arc(60*\x+\aq+180+30:60*\x+\aq+30:\qrad/2) arc (60*\x+60+\ap-30:60*\x+120+\ap-30:\prad) arc (60*\x+\aq+120-30:180+60*\x+\aq+120-30:\qrad/2) arc (180+60*\x+\aq+120-30+60:60*\x+\aq+120-30+60:\qrad/2);
%  }
% \end{tikzpicture}
 \vskip-0.5\baselineskip
 \caption{ A non-radially\protect\footnotemark,
  a radially, and a rotationally generated monohedral tiling of $\BB^2$.
  In contrast to these three,
  the rightmost, radially generated monohedral tiling is not rotationally invariant.
%  and radially generated subtilings of radially generated tilings of $\BB^2$.%
 }\label{fig:twelve}\vskip-1.75\baselineskip
\end{figure}\noindent%
\footnotetext{
It may be worth noting that this configuration also appears regularly in various places: this was chosen, for example, as the logo of the MASS program at Penn State University, it appears on the front page of five issues of the Hungarian problem-solving mathematical journal K\"oz\'episkolai Matematikai Lapok \cite{KoMaL}, and it can be found also in the book \cite[Figure C8]{CroftFalconerGuy1994}.}%

%{\ak We call a monohedral tiling of $\BB^2$ \emph{rotationally generated},
%if every tile is radially generated at their common point $O$.}
Our main result is the following.

\begin{thm}\label{thm:main}
 Any monohedral tiling of $\BB^2$ with at most three topological discs
 is rotationally generated. %In particular, we have $n(\BB^2) \geq 4$.
\end{thm}

This result
implies Conjecture~\ref{conj:radiallygenerated} for tilings with at most $3$ tiles,
yields the first nontrivial lower bound for $n(\BB^2)$,
and in particular
proves that the answer for Question~\ref{ques:Stein}
is refuting for tilings with at most three tiles.
    
\goodbreak
\begin{cor}
 We have $n(\BB^2) \geq 4$.
\end{cor}

\begin{cor}
 There is no monohedral tiling of $\BB^2$ with at most three topological discs as tiles
 such that the center of $\BB^2$ is contained in exactly one of them.
\end{cor}

In Section~\ref{sec:prelim} we introduce the notions used in the paper,
investigate the basic properties of monohedral tilings of $\BB^2$,
and prove a series of lemmas that we use in the proof of Theorem~\ref{thm:main}.
In Section~\ref{sec:proof} we prove Theorem~\ref{thm:main}.

Finally, in Section~\ref{sec:remarks}
we collect our additional remarks and propose some open problems.
%Moreover, we mention some very closely related problems and results,
%such as Goncharov's monohedral \emph{coverings} \cite{Goncharov},
%and the so-called \emph{divisibility} \cite{Richter,KissLaczkovich,KissSomlai}.

\section{Notations and preliminaries}\label{sec:prelim}

Throughout the proof, we denote by $\BB^2$
the closed unit circular disc with the origin $O=(0,0)$ as its center,
and its boundary by $\SS^1=\bd\BB^2$.
We say that two points $P,Q \in \SS^1$ are \emph{antipodal} if $d(P,Q) = 2$, where
$d(\cdot, \cdot)$ denotes Euclidean distance.
For points $P,Q \in \Re^2$, the closed segment with endpoints $P,Q$ is denoted by $\overline{PQ}$.

For any $P,Q \in \Re^2$ with $d(P,Q) \leq 2r$,
the \emph{$r$-spindle} $\ominus^r_{P,Q}$ of two points $P,Q$
is by definition (cf.\ \cite{BLNP} or \cite{FKV})
the intersection of all Euclidean discs of radius $r>0$
that contain~$P$ and~$Q$. In other words, $\ominus^r_{P,Q}$ is
the region bounded by the two circular arcs of radius $r>0$
that connect $P$ and $Q$ and are not longer than a half-circle.

A set homeomorphic to $\BB^2$ is called a \emph{topological disc}.
The boundary of a topological disc is a simple, closed, continuous curve,
called \emph{Jordan curve}.
On the other hand, the Jordan--Schoenflies theorem \cite{Schoenflies} yields
that every Jordan curve is the boundary of a topological disc.
We remark that since all topological discs are compact,
they are Lebesgue measurable; we denote their measure by $\area(\cdot)$.
Nevertheless, there are topological discs (see, e.g. the Koch snowflake,
or for more examples \cite{Sagan}) whose boundary is not rectifiable.
Our next lemma, which we use in the proof,
holds for these topological discs as well.

\begin{lem}\label{lem:Jordan}
Let $\Gamma$ be a Jordan curve and $\mathcal C$ be a simple, continuous curve.
Then $\Gamma$ contains finitely many congruent copies of $\mathcal C$
which are mutually disjoint, apart from possibly their endpoints.
\end{lem}

\begin{proof}
Assume for contradiction that $\Gamma$ contains infinitely many congruent
copies $\mathcal C_n$ ($n=1,2,\ldots$) of $\mathcal C$ which are mutually disjoint, apart from possibly their endpoints.
Let $P_n$ and $Q_n$ denote the endpoints of $\mathcal C_n$.
Since $\Gamma$ is compact, we may assume that $\lim_{n \to \infty} P_n = P$ and $\lim_{n \to \infty} Q_n =Q$ for some $P,Q \in \Gamma$.
By the properties of congruence, $P \neq Q$. On the other hand, since $\Gamma$ is homeomorphic to $\SS^1$, the congruent copies of $\mathcal C$
correspond to mutually nonoverlapping circular arcs on $\SS^1$. Clearly, this implies that $P=Q$, a contradiction.
\end{proof}

\begin{lem}\label{lem:topolofdissection}
 Let $\D = \D_1 \cup \D_2$,
 where $\D, \D_1$ and $\D_2$ are topological discs,
 and $\inter \D_1 \cap \inter \D_2 = \emptyset$.
 Then $\SS_1 = \D_1 \cap \bd \D$, $\SS_2 = \D_2 \cap \bd \D$ and $\SS=\bd \D_1 \cap \bd \D_2$
 are simple continuous curves.
\end{lem}

\goodbreak
\begin{proof}
As $\mathcal D$ is a topological disc, we have a homeomorphism $\chi$
such that $\chi(\mathcal D)=\BB^2$.
Since the statement of the lemma is topologically invariant,
it is sufficient to prove it in the case $\D = \B^2$.
Thus, we may assume that $\SS_i=\SS^1\cap\mathcal D_i$ for $i=1,2$,
where we observe that since $\mathcal D_i$ and $\SS^1$ are closed, so is $\SS_i$.

First, we show that $\SS_1$ and $\SS_2$ are connected.
Assume, for example, that some $X_1, Y_1 \in \SS_1$
cannot be connected by an arc in $\SS_1$.
Then there are some points $X_2, Y_2 \notin \SS_1$
that separate $X_1$ and $Y_1$ in $\SS^1$.
Clearly, we have $X_2, Y_2 \in \SS_2$.
For any $i=1,2$, since $\mathcal D_i$ is a topological disc,
there is a simple, continuous curve $\gamma_i$ with endpoints $X_i, Y_i$ such that
apart from these points $\gamma_i$ is contained in $\inter \SS_i$. 
By continuity, $\gamma_1 \cap \gamma_2 \neq \emptyset$,
implying that $\inter \D_1 \cap \inter \D_2 \neq \emptyset$, a contradiction.
Thus, $\SS_1$ and $\SS_2$ are connected,
which yields that they are closed circular arcs in $\SS^1$.
Let the (common) endpoints of these arcs be $P$ and $Q$.

The points $P,Q\in\SS_1\cap\SS_2$ are also in
$\bd\mathcal D_1\cap\bd\mathcal D_2$, hence they are connected by
a simple continuous curve in $\bd\mathcal D_1\setminus\SS_1$
and also in $\bd\mathcal D_2\setminus\SS_2$.
These curves coincide because $\mathcal D=\mathcal D_1\cup\mathcal D_2$,
hence it is $\SS$, and the proof of Lemma~\ref{lem:topolofdissection} is complete.
\end{proof}

\begin{lem}\label{lem:3tilestopology}
Let $\{ \D_1, \D_2, \D_3 \}$ be a tiling of the topological disc $\D$
where for $i=1,2,3$, $\D_i$ is a topological disc such that $\SS_i=\D_i \cap \bd \D$
is a nondegenerate simple continuous curve.
Then $\D_1 \cap \D_2 \cap \D_3$ is a singleton $\{ M \}$,
and for any $i \neq j$, $\D_i \cap \D_j$ is
a simple continuous curve connecting $M$ and a point in $\bd \mathcal D$.
\end{lem}

\begin{proof}
Suppose for contradiction that there are
two distinct points $M_1, M_2 \in \mathcal D_i$ for $i=1,2,3$.
For any $i$, let $\Gamma_i$ be a simple,
continuous curve connecting $M_1$ and $M_2$
which is contained in $\inter \mathcal D_i$,
apart from $M_1$ and $M_2$.
Note that for any $i \neq j$, $\Gamma_i \cup \Gamma_j$
is a simple, closed, continuous curve.
Thus, the union of a pair of the curves, say $\Gamma_1 \cup \Gamma_2$
encloses the third one.
This implies that $\Gamma_1 \cup \Gamma_2$ encloses $\mathcal D_3$.
Since $M_1, M_2 \notin \SS^1$ by our conditions, it follows that
$\mathcal D_3$ is disjoint from $\SS^1$; a contradiction.
Thus, $\mathcal D_1 \cap \mathcal D_2 \cap \mathcal D_3$
contains at most one point.
On the other hand, since the closure $\mathcal{X}_1$ of $(\bd \D_1) \setminus \SS_1$
is a simple, connected curve and it can be decomposed into the closed sets
$\mathcal{X}_1 \cap \D_2$ and $\mathcal{X}_1 \cap \D_3$,
it follows that these sets intersect,
that is $\D_1 \cap \D_2 \cap \D_3$ is not empty.
Thus, $\D_1 \cap \D_2 \cap \D_3 = \{ M \}$ for some $M \in \inter \D$.

To prove the second part of Lemma~\ref{lem:3tilestopology}, we may apply an argument like in the proof of Lemma~\ref{lem:topolofdissection}.
\end{proof}

By the \emph{circumcircle} of a topological disc $\mathcal D$
we mean the unique smallest closed Euclidean circle encompassing $\mathcal D$.
The convex hull of the circumcircle is the \emph{circumdisc} of $\mathcal D$,
the radius of the circumcircle is the \emph{circumradius} of $\D$.
Observe that the center of the circumcircle $\mathcal C$ of $\mathcal D$
is in $\conv(\mathcal C\cap\mathcal D)$, as otherwise a smaller
circle would encompass $\mathcal D$.

\begin{lem}\label{lem:circumcircle}
 Assume that $\SS^1$ is the common circumcircle of the non-overlapping
 congruent topological discs $\mathcal D_1$ and $\mathcal D_2$.
 Then there is a diameter $\overline{PQ}$ of $\BB^2$ separating
 $\SS_1 = \D_1 \cap \SS^1$ and $\SS_2 = \D_2 \cap \SS^1$.
 Furthermore, any congruence $g$ with $g(\mathcal D_1) = \mathcal D_2$
 is either the reflection about the line of $\overline{PQ}$,
 or the reflection about $O$.
%$\conv\SS_1$ and $\conv\SS_2$ are in the half discs of $\BB^2$ dissected by the diameter $\conv\SS_1\cap\conv\SS_2$.
\end{lem}

%\goodbreak
\begin{proof}
Using the idea of the proof of Lemma~\ref{lem:topolofdissection},
it follows that there are no pairs of points $X_1, Y_1 \in\SS_1$ and
$X_2, Y_2 \in\SS_2$ that strictly separate each other on $\SS^1$.
In other words, there is a line $\ell$ separating $\SS_1$ and $\SS_2$.
On the other hand, as $O\in\conv\SS_1\cap\conv\SS_2$, $\ell$ contains $O$
and $\ell \cap \SS^1 \subseteq \SS_1 \cap \SS_2$,
proving the first statement with $\{ P,Q \} = \ell \cap \SS^1$.
We note that from this argument it also follows that $\SS_1 \cap \SS_2 = \{ P,Q \}$.

Consider some isometry $g$ with $g(\mathcal D_1) = \mathcal D_2$.
The uniqueness of the circumcircle clearly implies that
$g(\SS^1) = \SS^1$, and thus, $g (\{ P, Q \}) = \{ P, Q \}$.
This implies that $g$ is either the reflection about the line of $\overline{PQ}$,
the reflection about the line bisecting $\overline{PQ}$, or the reflection about $O$.
We show that the conditions of the lemma exclude the second case:
Consider a simple, continuous curve $\Gamma$ from $P$ to $Q$ such that
$\Gamma \setminus \{ P, Q \} \subset \inter \mathcal D_1$.
Then at least one point $R$ of $\Gamma$ lies on the line $\ell^{\perp}$
bisecting $\overline{PQ}$. If $g$ is the reflection about $\ell^{\perp}$,
then $g(R) = R$, and hence, $R \in \inter \mathcal D_1 \cap \inter \mathcal D_2$;
a contradiction.
\end{proof}

In the remaining part of Section~\ref{sec:prelim},
we deal only with a monohedral tiling of $\BB^2$,
where the tiles $\D_i$, $i=1,2,\ldots,n$,
are congruent copies of a topological disc $\D$.
For any $j \neq 1$, we fix an isometry $g_{1j}$ mapping $\D_1$ into $\D_j$,
and for any values of $i,j$, we set $g_{ij}=g_{1i}^{-1} \circ g_{1j}$.
Then, by definition, we have $g_{ji}=g_{ij}^{-1}$ for all values of $i,j$.
Finally, we set $\SS_i = \D_i \cap \SS^1$ for all values of $i$.
%Observe that $g_{ij}=g_{ji}^{-1}$,
%because $g_{ij}(\mathcal D_i)=\mathcal D_j$ and $g_{ji}(\mathcal D_j)=\mathcal D_i$.

\begin{lem}\label{lem:nodiameter}
 If $\D$ contains two points at the distance $2$, then $n=1$ or $n=2$,
 and the tiling is rotationally generated.
\end{lem}

\begin{proof}
If $\mathcal D$ contains two points at the distance $2$,
then each tile contains two antipodal points of $\BB^2$.
Thus, $\BB^2$ is the circumdisc of each tile,
which implies that $g_{ij}(\BB^2)=\BB^2$ for all values of $i,j$.
Since $O \in \D_i$ for some value of $i$,
it also yields that $O \in \D_i$ for all values of~$i$.
Then, by Lemma~\ref{lem:circumcircle},
there is a diameter $\overline{PQ}$ of $\BB^2$
whose endpoints belong to every tile,
and the congruence between any two of them is either a reflection
about the line through $\overline{PQ}$,
or the reflection about the midpoint of $\overline{PQ}$.
This implies that there are at most two tiles.

To prove that the tiling is rotationally generated,
assume that $n=2$, and $\D_2$ is a reflected copy of $\D_1$
about the line through $\overline{PQ}$.
Since in this case $\D_1$ and $\D_2$ are the two closed half discs of $\BB^2$
containing $\overline{PQ}$ in their boundaries, the statement follows.
\end{proof}

\begin{lem}\label{lem:only1arc}
 For all values of $i$, $\SS_i$ ($i=1,\dots,n$) is a closed,
 connected arc in $\SS^1$.
\end{lem}

%\goodbreak
\begin{proof}
As $\SS_1$ is compact, there are points $P,Q\in \SS_1$ farthest from each other in $\SS_1$.
If $P,Q$ are antipodal points of $\SS^1$, then
every $\mathcal D_i=g_{1i}(\mathcal D_1)$ ($i=1,\ldots$) contains antipodal points,
hence $\BB^2$ is the circumdisc of every tile.
Then Lemma~\ref{lem:nodiameter} yields that $n=1$ or $n=2$.
The case $n=1$ is trivial, and if $n=2$, then
by Lemma~\ref{lem:circumcircle}, there is a diameter $\overline{PQ}$ separating $\SS_1$ and $\SS_2$, which implies that
$\SS_1$ and $\SS_2$ are closed half-circles.
Thus, we may assume that $P,Q$ are not antipodal.

Let $\Gamma\subset\SS^1$ be the shorter arc connecting $P$ and $Q$. We show that $\Gamma \subset \D_1$.

For contradiction, suppose that a point $X \in \Gamma$ does not belong to $\mathcal D_1$.
Then, without loss of generality, we may assume that $X \in \D_2$, and that $X \neq P$, $X \neq Q$.

Let $r>0$ be the radius of the circumdisc  $\BB$ of $\mathcal D_1$.
Since $\D_1$ is compact, and it does not contain antipodal points of $\SS^1$, we have $r < 1$, implying that
$\ominus^r_{P,Q}$ contains $\Gamma\setminus\{P,Q\}$ in its interior.
Thus $\Gamma \subset \BB$, and $\Gamma\setminus\{P,Q\}\subset\inter\BB$.
Let $\Gamma'$ be a continuous curve connecting $P$ and $Q$ such that $\Gamma' \setminus \{P,Q\} \subset \inter\mathcal D_1$.
%hence $\Gamma' \setminus \{P,Q\} \subset \inter\BB$.
This yields that $\Gamma \cup \Gamma'$ is a simple, continuous, closed curve in $\BB$ enclosing $\D_2$.
which, by the congruence of $\D_1$ and $\D_2$, implies that the $\BB$ is the circumdisc of $\D_2$ as well.
Hence, by Lemma~\ref{lem:circumcircle} it follows that 
$\conv(\bd \BB\cap\mathcal D_1 \cap \mathcal D_2)$ is a diameter $\delta$ of $\BB$.
As $P, Q$ are the only points of $\Gamma \cup \Gamma'$ that may fall on $\bd \BB$,
we have $\delta = \overline{PQ}$.

From Lemma~\ref{lem:circumcircle} it also follows that $g_{12}$ is the reflection about the line of $\overline{PQ}$,
or the reflection about the midpoint of $\overline{PQ}$, and in particular $g_{21}=g_{12}$.
On the other hand, observe that $g_{12} ( \ominus^1_{P,Q} ) = \ominus^1_{P,Q} = \BB^2 \cap g_{12}(\BB^2)$.
Since $\D_1 \subset \BB^2$ and $\D_1 = g_{12}(\D_2) \subset g_{12}(\BB^2)$,
this implies that $\D_1, \D_2 \subset \ominus^1_{P,Q}$.
Now, if there is a point $R\in \D_2 \cap \bd \ominus^1_{P,Q}\setminus \Gamma$
then $R$ and $X$ can be connected with a continuous curve in $\inter \D_2$,
while $P,Q\in\inter\D_1$, a contradiction.
Hence $ \D_2 \cap \bd \ominus^1_{P,Q}\subset \Gamma$,
and accordingly $\mathcal D_1\cap\SS^1=\{P,Q\}$,
and, in particular, $\mathcal D_1\cap\Gamma=\{P,Q\}$.

Assume that there is an interior point $Y$ of $\Gamma$
that belongs to, say, $\D_3$. Since $P,Q \in\mathcal D_2$,
we may repeat the argument in the previous paragraph,
replacing $\mathcal D_1$ and $\mathcal D_2$ by $\mathcal D_2$ and $\mathcal D_3$, respectively,
and obtain that $\mathcal D_2\cap\Gamma=\{P,Q\}$ contradicting
our assumption that there is an interior point $X\in\mathcal D_2$ of $\Gamma$.
Thus, $\Gamma\subset\mathcal D_2$, which yields by Lemma~\ref{lem:circumcircle} that
$\mathcal D_2 \cap \SS^1 = \Gamma$ and $\ominus^1_{P,Q}=\mathcal D_1 \cup\mathcal D_2$.
From this, in particular, it follows that
$\area (\mathcal D_1 ) = \area (\mathcal D_2 ) = {\area(\ominus^1_{P,Q})}/{2}$.

Since for all values of $i$, $g_{2i}(\{ P,Q \}) \subset \BB^2$,
the definition of $1$-spindle implies that
$\mathcal D_i  \subset \ominus^1_{g_{2i}(P),g_{2i}(Q)} \subset \BB^2$, and $g_{2i}(\mathcal \D_2) \setminus \Gamma \subset \inter \ominus^1_{g_{2i}(P),g_{2i}(Q)}$ is disjoint from $\SS^1$. In other words, the sets $g_{2i}(\Gamma)$ cover $\SS^1$. 
Note that these arcs may intersect each other only at their endpoints,
and if $|\SS^1\cap g_{2i}(\Gamma)|\geq 3$, then $g_{2i}(\Gamma)\subset\SS^1$.
Thus, $\SS^1$ can be decomposed into
finitely many, say $k < n$ circular arcs, each of which is congruent to $\Gamma$.

Let $s={2\pi}/{k}$ denote the arclength of $\Gamma$.
Then $ks=2\pi$ on one hand, and
\[
 \frac{\pi}{n}=\frac{\area(B^2)}{n}=\area(\mathcal D_2)
              =\frac{\area(\ominus^1_{P,Q})}{2}=\frac{s-\sin s}{4}
              =\frac{\frac{2\pi}{k}-\sin\frac{2\pi}{k}}{4}
\]
on the other hand. Thus, we have $\sin\frac{2\pi}{k}=\pi(\frac{2}{k}-\frac{4}{n})$.
The left-hand side is an algebraic number (see, e.g. \cite[Theorem~2.1]{Tangsupphathawat2014}),
from which $\frac{2}{k}=\frac{4}{n}$ follows, hence $\sin\frac{2\pi}{k} = 0$, implying that $k$ is a divisor of $2$, contradicting our assumption
that $\Gamma$ is shorter than a half-circle.
\end{proof}

\begin{rem}\label{rem:convhullarc}
Since $\SS_i \subset \bd \conv \D_i$, it follows that for all values of $i,j$, we have $g_{ij}(\SS_i) \subset \bd \conv \D_j$.
\end{rem}

\begin{rem}\label{rem:disjoint}
For any values of $i,j,k$, the arcs $g_{ik}(\SS_i)$ and $g_{jk}(\SS_j)$ share at most some of their endpoints, or they coincide.
\end{rem}

\begin{proof}
Observe that since for any $i$, $\SS_i$ is contained in the convex hull of $\mathcal D_i$,
if the arcs $g_{ik}(\SS_i)$ and $g_{jk}(\SS_j)$ are not disjoint apart from (possibly) their endpoints,
then $g_{ik}(\SS_i) \cap g_{jk}(\SS_j)$ is a nondegenerate unit circular arc, and thus, 
$g_{ik}(\SS_i) \cup g_{jk}(\SS_j)$ lies on a unit circle $\SS$.

Since $\SS_i = g_{ki} (g_{ik}(\SS_i)) \subset \SS^1$, we have
$g_{ki} (\SS) = \SS^1$. Thus, $\SS_i \cup g_{ki} (g_{jk}(\SS_j)) \subset \SS^1 \cap \mathcal D_ i = \SS_i$.
This implies that $g_{jk}(\SS_j) \subseteq g_{ik}(\SS_i)$. The containment relation $g_{ik}(\SS_i) \subseteq g_{jk}(\SS_j)$ can be obtained using a similar argument, which yields the desired equality.
\end{proof}

\begin{lem}\label{lem:allarc}
 Let $\mathcal D_1, \mathcal D_2, \dots,\mathcal D_n$
 be a monohedral tiling of $\BB^2$, where $n > 1$.
 Then at least two of the arcs $\SS_1$, $g_{21}(\SS_2)$,
 $\dots$, $g_{n1}(\SS_n)$ coincide.
\end{lem}

\begin{proof}
 Suppose for contradiction that the arcs
 $\SS_1, g_{21}(\SS_2), \ldots, g_{n1}(\SS_n)$
 are disjoint apart from possibly their endpoints.
 By our earlier observation these arcs are in $\bd\conv\mathcal D_1$.
 As the total turning angle of these $n$ arcs is $2\pi$,
 and the total turning angle along the boundary of a convex body is also $2\pi$,
 $\bd\conv\mathcal D_1$ may only consist in excess of these arcs some segments
 that connect the endpoints of these arcs in a smooth way.
 In other words,
 $\conv\mathcal D_1=\mathcal P+\BB^2$ for some convex $n$-gon $\mathcal P$.
 This implies that the circumradius of $\mathcal D_1$ is at least $1$,
 with equality if and only if $\mathcal D_1 =\BB^2$, a contradiction.
\end{proof}

\begin{defn}\label{def:multicurve}
 A \emph{multicurve}  (cf.\ also \cite{KurusaMF}) is
 a finite family of continuous simple curves,
 called the \emph{members of the multicurve},
 which are parameterized on non-degenerate closed finite intervals,
 and any point of the plane belongs to at most one member,
 or it is the endpoint of exactly two members.
 If $\mathcal{F}$ and $\mathcal{G}$ are multicurves,
 $\bigcup \mathcal{F}= \bigcup \mathcal{G}$,
 and every member of $\mathcal{F}$ is the union of some members of $\mathcal{G}$,
 we say that $\mathcal{G}$ is a \emph{partition} of $\mathcal{F}$.
\end{defn}

\begin{defn}\label{defn:equidecomposable}
 Let $\mathcal{F}$ and $\mathcal{G}$ be multicurves.
 If there are partitions $\mathcal{F}'$ and $\mathcal{G}'$ of $\mathcal{F}$
 and $\mathcal{G}$, respectively, and a bijection
 $f\colon\mathcal{F}' \to \mathcal{G'}$ such that $f(\C)$ is
 congruent to $\C$ for all $\C \in \mathcal{F}'$,
 we say that $\mathcal{F}$ and $\mathcal{G}$ are \emph{equidecomposable}.
\end{defn}

The following lemma can be proved very similarly to the analogous statement
for equidecomposability of polygons \cite{Frederickson}, thus we omit the proof. 

%\goodbreak
\begin{lem}\label{lem:equirel}
 Equidecomposability is an equivalence relation on the family of multicurves in $\Re^2$.
\end{lem}

\begin{cor}\label{cor:graphs}
 If $\mathcal{F}$ and $\mathcal{G}$ are multicurves
 with $\bigcup \mathcal{F} = \bigcup \mathcal{G}$,
 then $\mathcal{F}$ and $\mathcal{G}$ are equidecomposable.
\end{cor}

\begin{proof}
Clearly, it is sufficient to prove the statement
for the connected components of $\bigcup \mathcal{F}$,
and by Lemma~\ref{lem:equirel} we may assume that one of the multicurves,
say $\mathcal{G}$, is a simple continuous curve.
But then $\mathcal{F}$ is a partition of $\mathcal{G}$,
in which case the statement is obvious.
\end{proof}

\begin{cor}\label{cor:cutcong}
If $\mathcal{F}$ and $\mathcal{G}$ are equidecomposable,
and their subfamilies $\mathcal{F}' \subseteq \mathcal{F}$
and $\mathcal{G}' \subseteq \mathcal{G}$ are equidecomposable,
then $\mathcal{F} \setminus \mathcal{F}'$ and $\mathcal{G} \setminus \mathcal{G}'$
are equidecomposable.
\end{cor}

\begin{proof}
By Lemma~\ref{lem:equirel}, we may assume that $\bigcup \mathcal{F} = \bigcup \mathcal{G}$. 
Without loss of generality,
we may also assume that $\bigcup \mathcal{F}$ is connected,
which yields that we may regard both $\mathcal{F}$ and $\mathcal{G}$
as different partitions of the same simple, continuous curve.
More specifically, after reparametrizing if necessary,
we may assume that there is some curve $\C\colon[a,b] \to \Re^2$,
and partitions $P_F$ and $P_G$ of $[a,b]$ such that the elements of $\mathcal{F}$
and $\mathcal{G}$ are the restrictions of $\C$ to the subintervals of $P_F$ and $P_G$,
respectively.
By Corollary~\ref{cor:graphs},
a multicurve is equidecomposable with any of its partitions,
and hence, we may assume that $P_F = P_G$,
and there is a bijection between the elements of $\mathcal{F'}$ and $\mathcal{G'}$
such that the corresponding elements are congruent.
Since congruence is an equivalence relation,
it is clear that any such bijection can be extended to all subintervals of $P_F$,
which proves the assertion.
\end{proof}

\section{Proof of Theorem~\ref{thm:main}}\label{sec:proof}

First, consider a monohedral tiling of $\BB^2$
with the topological discs $\D_1$ and $\D_2$.
The containment $O \in \D_1 \cap \D_2$ can be proved
in a number of elementary ways (see, e.g. \cite{Kanelbelov2002});
here we also show that the tiling is rotationally generated.

By Lemma~\ref{lem:only1arc}, for $i=1,2$,
$\SS_i = \D_i \cap \SS^1$ is a connected arc and hence,
$\SS_1$ or $\SS_2$ is an arc of length at least $\pi$.
Thus, $\mathcal D_1$ or $\mathcal D_2$
contains a pair of antipodal points of $\BB^2$,
which, by Lemma~\ref{lem:nodiameter},
implies that the tiling is rotationally generated.

From now on,
we consider the case that $\BB^2$ is decomposed
into three congruent topological discs $\D_1, \D_2, \D_3$, and for $i=1,2,3$,
we set $\SS_i = \D_i \cap \SS^1$.
By Lemmas~\ref{lem:only1arc} and \ref{lem:nodiameter},
we may assume that each tile intersects $\SS^1$
in a nondegenerate cicle arc, which is smaller than a half-circle.

By Lemma~\ref{lem:3tilestopology},
we have that $\D_1 \cap \D_2 \cap \D_3$ consists of a single point $M \in \inter \BB^2$,
and that for any $i \neq j$, $\D_i \cap \D_j$ is a simple,
continuous curve connecting $M$ and a point of $\SS^1$.

To prove the assertion, we distinguish some cases.
Before we do it, we observe that by Remark~\ref{rem:disjoint},
any pair of the curves $\SS_1$, $g_{21}(\SS_2)$ and $g_{31}(\SS_3)$
intersect in at most a common endpoint, or they coincide.

\noindent
\emph{\textbf{Case 1}:
No pair of the arcs $\SS_1$, $g_{21}(\SS_2)$ and $g_{31}(\SS_3)$ coincide.}%
\\
In this case we immediately have a contradiction by Lemma~\ref{lem:allarc}.

\noindent
\emph{\textbf{Case 2}:
Two of the arcs  $\SS_1$, $g_{21}(\SS_2)$ and $g_{31}(\SS_3)$ coincide,
the third one is different.}%
\\
Using a suitable relabeling of the tiles, we may assume that $\SS_1 = g_{21}(\SS_2)$.
Let the arclength of this arc be $0 < \alpha < \pi$,
and the arclength of $\SS_3$ be $\beta$.
The equality $\SS_1 = g_{21}(\SS_2)$ implies, in particular,
that $g_{21}$ is an isometry of $\SS^1$; or more generally that
it is either the reflection about the symmetry axis $\ell$ of
$\SS_1 \cup\mathcal  S_2$ or a rotation around $O$ with angle $\alpha$.
We may assume without loss of generality that $\ell$ is the $y$-axis,
the common point of $\SS_1$ and $\SS_2$ is $(0,1)$,
and $\SS_1\subset\{(x,y):x\le0\}$.
Furthermore, in the proof we set
$\C_1 = \D_1 \cap \D_3$, and $\C_2 = \D_2 \cap \D_3$.

\noindent
\emph{\textbf{Subcase 2.a}:
$g_{21}$ is the reflection about $\ell$.}%
\\
If there is a point $P\in\inter\mathcal D_1\cap\{(x,y):x>0\}$,
then a continuous curve $\Gamma$ in $\inter\mathcal D_1$
connects $P$ and the midpoint of $\SS_1$,
so $g_{12}(\Gamma)$ connects the midpoint of $\SS_2$
to $g_{12}(P)$ in $\inter\mathcal D_2$.
This implies that $\Gamma\cap g_{12}(\Gamma)$ %\subset\ell$
is in $\inter\mathcal D_1\cap\inter\mathcal D_2=\emptyset$, which is a contradiction.
Thus we have $\mathcal D_1\subset\{(x,y):x\le0\}$ and also
$\mathcal D_2\subset\{(x,y):x\ge0\}$.

Observe that
$g_{13}(\SS_1)=g_{23}(g_{12}(\SS_1)) = g_{23}(\SS_2)$,
and
$\mathcal D_3=\cl(\mathcal \BB^2 \setminus (\mathcal D_1 \cup \mathcal D_2))$
is symmetric in $\ell$.
We denote this arc of length $\alpha$ by $\SS=g_{13}(\SS_1)$.
Note that by the conditions of  Case 2 $\SS \neq \SS_3$, and
$\SS \subset \bd \conv \mathcal D_3$ by Remark~\ref{rem:convhullarc}.
Furthermore, $\bd \conv \mathcal D_3$ does not contain any arc of length $\alpha$
apart from $\SS$ and possibly $\SS_3$, as otherwise the idea of the proof of
Lemma~\ref{lem:allarc} yields a contradiction.
Thus, $\SS$ is symmetric in the $y$-axis.
 
Since $\mathcal D_3$ is connected, and every point of $\ell$ belongs
either to $\mathcal D_3$, or to both $\mathcal D_1$ and $\mathcal D_2$,
the segment connecting the midpoint $X$ of $\SS$ and
the midpoint $Y$ of $\SS_3$ belongs to $\mathcal D_3$.
Let the length of $\overline{XY}$ be $\delta > 0$, and note that
the fact $X,Y \in \bd \conv \D_3$ yields
that the line through $\overline{XY}$ intersects $\D_3$
exactly in $\overline{XY}$
and $\overline{XY} \setminus \{ X,Y \} \subset \inter \D_3$.
For $i=1,2$, $g_{3i}(\overline{XY})$ is
the segment of length $\delta$ in $\BB^2$,
starting at the midpoint of $\SS_i$, and perpendicular to it.
%From this observation and the fact that $\SS, \SS_3 \subset \bd \conv \mathcal D_3$, it follows that if
Thus, if $\delta < 1$, then $O \notin \mathcal D_i$ for any value of $i$,
if $\delta > 1$, then $O \in \inter \mathcal D_i$ for all values of $i$,
and if $\delta = 1$, then $O$ is the midpoint of a unit circle arc
in the boundary of each of the $\D_i$s, which is a contradiction.

\noindent
\emph{\textbf{Subcase 2.b}:
 $g_{21}$ is the rotation around $O$ by angle $\alpha$ in counterclockwise direction.}%
\\
As $O$ is a fixed point of $g_{21}$,
it follows that either $O \in \mathcal D_1 \cap \mathcal D_2$,
or $O \notin \mathcal D_1 \cup \mathcal D_2$.
By the definition of tiling and our assumptions,
in the first case $O \in \bd \mathcal D_1 \cap \bd \mathcal D_2$,
and in the second case $O \in \inter \mathcal D_3$.

First, consider the case that $O \in \bd\mathcal D_1 \cap \bd\mathcal D_2$.

Recall that by Lemma~\ref{lem:3tilestopology},
$D_1 \cap D_2 \cap D_3$ is a single point $M$,
and for any $i \neq j$, $\D_i \cap \D_j$ is a simple continuous curve connecting $M$
to a point of $\SS^1$. Thus, if $O=M$, then $g_{21}(\D_1 \cap \D_2) = \D_1 \cap\D_3$,
and $g_{12}(\D_1 \cap \D_2) = \D_2 \cap \D_3$.
Since $\bd \D_1$ and $\bd \D_3$ are equidecomposable,
this implies that $\SS_1$ and $\SS_3$ are congruent,
and hence $\alpha = {2\pi}/{3}$.
In other words, if $O=M$, then the tiling is rotationally generated.
Thus, we assume that $O \notin \D_3$,
which by the compactness of $\D_3$ yields the existence
of a small closed circular disc $\BB$ centered at $O$
such that $\BB \cap \D_3 = \emptyset$.
Let $t\mapsto \C(t)$ be a continuous parameterization of
the curve  $\D_1 \cap \D_2$ satisfying $O=\C(0)$,
and let $t_{+}=\sup\{t:\C([0,t]))\subset \BB\}$
and $t_{-}=\inf\{t:\C([t,0])\subset \BB\}$.
Then $g_{12}(\C(t_\pm))=\C(t_\mp)$,
which implies that $g_{12}$ is the reflection about $O$.
Thus $\alpha=\pi$ and $\beta = 0$,
which contradicts our assumptions.

In the remaining part of Subcase 2.b, we assume that $O\in \inter \mathcal D_3$.

Let $M_1=g_{21}(M)$ and $M_2=g_{12}(M)$.
Since $\alpha > 0$, we have $M_2 \neq M$.
On the other hand, we clearly have $M_2 \in \bd \D_2$ and $M_2 \notin \SS^1$.

Let $\BB$ be the circular disc in $\D_3$
that is centered at $O$ and is of maximum radius $r>0$.
Then $\BB$ is tangent to at least one of the curves $\mathcal C_1$ and $\mathcal C_2$,
say $\mathcal C_2$ touches $\BB$ in $X_2\in\bd\BB\cap\mathcal C_2$.
Let $X_1=g_{21}(X_2)$.
Then $X_1 \in \BB\cap \mathcal D_1=\BB\cap \mathcal C_1$ clearly,
hence $X_2\in g_{12}(\mathcal C_1)\cap \mathcal C_2\neq \emptyset$.
%which implies $g_{12}(M)=M_2\in \mathcal C_2$.
Since $g_{12}(\C_1)$ is a continuous curve in $\bd \D_2$,
connecting the intersection point of $\SS_1$ and $\SS_2$ to $M_2$ in $\inter \BB^2$,
it follows that $M \in g_{12}(\C_1)$, that is, $M_1\in \mathcal C_1$,
implying also $M_2 \in \C_2$.

Thus, $M_1$ divides the curve $\mathcal C_1$ into two parts:
one from $M$ to $M_1$, which we denote by $\mathcal C_1^M$,
and the other one from $M_1$ to a point of $S_1$,
which we denote by $\mathcal C_1^S$.
We define the parts $\C_2^M$ and $\C_2^S$ of $\C_2$ similarly,
using $M_2$ in place of $M_1$.
Furthermore, we set $\mathcal C_3^S=\mathcal D_1 \cap \mathcal D_2$.
 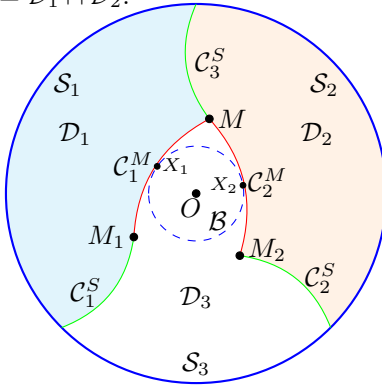
\begin{figure}[H]\vskip-1.2\baselineskip\centering
   \begin{tikzpicture}[x=-25mm,y=25mm,
    pont/.style={circle,draw=black,fill=black,inner sep=1.0pt}]
    \coordinate (O) at (0,0);

    \coordinate (M) at (100:0.4);
    \coordinate (N) at ( 90:1);

    \coordinate (S2) at ( -45:1);
    \coordinate (S1) at (-135:1);

    \coordinate (M2) at ( -35:0.4);
    \coordinate (M1) at ( 235:0.4);

    \fill[orange!10] (M) to[bend left] (N) arc (90:225:1) to[bend right] (M1) to[bend right=29] (M);

    \fill[cyan!10] (M) to[bend left] (N) arc (90:-45:1) to[bend right] (M2) to[bend left=29] (M);

    \draw[blue,thick] (O) circle (1);
    \draw[blue,thin,dashed] (O) circle (0.25);

    \draw[bend left,green] (M) to node[pos=0.5,right,black]{$\mathcal C_3^S$} (N);
    \draw[bend left,green] (M1) to
     node[pos=0.5,right,black]{$\mathcal C_2^S$} (S1);
    \draw[bend left,green] (M2) to 
     node[pos=0.5,left,black]{$\mathcal C_1^S$} (S2);

    \draw[bend left=29,red] (M) to
         node[pos=0.5,right,xshift=-2,black]{$\mathcal C_2^M$} (M1);
    \draw[bend right=29,red] (M) to
         node[pos=0.5,left,xshift=3,black]{$\mathcal C_1^M$} (M2);

    \node[pont,label={right,xshift=-2:$M$}] at (M) {};
    \node[pont,label={right,xshift=-2,yshift=2:$M_2$}] at (M1) {};
    \node[pont,label={left,xshift=2:$M_1$}] at (M2) {};
    
    \node[pont,inner sep=0.75pt,label={right,xshift=-3:\scriptsize$X_1$}]
     (X2) at (35:0.25) {};
    \node[pont,inner sep=0.75pt,label={left,xshift=3:\scriptsize$X_2$}]
     (X1) at (35+135:0.25) {};

    \node[pont,label={below,xshift=-2,yshift=3:$O$}] at (O) {};
    \node[label={above,xshift=2,yshift=-4:$\BB$}] at (-110:0.25) {};

    \node[label={right:$\mathcal D_2$}] at (145:0.55) {};
    \node[label={left:$\mathcal D_1$}] at ( 35:0.55) {};
    \node[label={above,yshift=-4:$\mathcal D_3$}] at (-90:0.65) {};

    \node[label={left,xshift=2:$\SS_2$}] at (145:1.0) {};
    \node[label={right,xshift=-2:$\SS_1$}] at ( 35:1.0) {};
    \node[label={above,yshift=-4:$\SS_3$}] at (-90:1.0) {};
    
   \end{tikzpicture}\vskip-.85\baselineskip
   \caption{$\BB^2$ is dissected into three topological discs.}%
   \label{fig:threeparts}\vskip-0.75\baselineskip
 \end{figure}
We clearly have
$g_{21}(\mathcal C_2^M) = \mathcal C_1^M$, $g_{21}(\mathcal C_2^S)=\mathcal C_3^S$
and $g_{21}(\mathcal C_3^S)=\mathcal C_1^S$.
Observe that since $\D_1, \D_2$ and $\D_3$ are congruent,
their boundaries are equidecomposable.
Furthermore, as $\mathcal C_1^S$, $\mathcal C_2^S$, and $\mathcal C_3^S$,
and also $\C_1^M$ and $\C_2^M$ are congruent,
we obtain by Corollary~\ref{cor:cutcong}
that $\SS_1$ and $\C_1^M \cup \SS_3$ are equidecomposable.
Thus we deduce that $\mathcal C_1^M$ (and also $\mathcal C_2^M$)
is a multicurve such that its every member curve
is a unit circular arc, and their total length is $\alpha - \beta \geq 0$.
%Without loss of generality, we assume that all these circular arcs are of
%length at least $\varepsilon$.

If a unit circular arc $\C$ is contained in the boundary of a tile $\D_i$,
it may happen that the convex side of $\C$ belongs to $\inter \D_i$,
and the concave side of $\C$ does not belong to $\D_i$.
In this case we say that $\C$ is a \emph{convex circular arc} of $\D_i$,
and in the opposite case that it is a
\emph{concave circular arc} of $\D_i$.
Clearly, if $\C$ is a unit circular arc in $\D_i \cap \D_j$ for some $i \neq j$,
then it is a convex circular arc of exactly one of $\D_i$ and $\D_j$.
Let $x$ and $y$ denote the total length of
the convex and concave unit circular arcs of $\D_1$ in $\C_1^M$.
Since $\C_1^M$ and $\C_2^M$ are congruent,
the total length of the convex and concave unit circular arcs of $\D_2$ in $\C_2^M$
is also $x$ and $y$, respectively.
Thus, the total length of the convex and concave unit circular arcs
of $\D_3$ in $\mathcal C_1^M\cup\mathcal C_2^M$ is $2y$ and $2x$, respectively.

%\goodbreak
The congruence of the tiles $\D_i$ and the curves $\C_i^S$ for $i=1,2,3$
yields that the total lengths of the convex and concave unit circular arcs
of $\D_1$ in $\SS_1 \cup \C_1^M$ is equal to the total lengths of
the convex and the concave unit circular arcs of $\D_3$
in $\SS_3 \cup \C_1^M \cup \C_2^M$, respectively.
This equality for convex circular arcs implies that $\alpha + x = \beta + 2y$,
and the equality for concave arcs implies $y = 2x$.
From these equations it follows that $x = (\alpha-\beta)/3$ and $y=2(\alpha-\beta)/3$.
Thus, in particular, it follows that if $\beta = \alpha$,
then $x=y=0$ and $M = M_1 =M_2$, which yields that $\alpha = 0$,
a contradiction.
This means that $\beta < \alpha$.

We show that $M$ is not an interior point of a unit circular arc in $\bd \D_3$
longer than $\alpha-\beta$.
Suppose for contradiction that $M$ is an interior point of such a circular arc $\C$.
If one of $M_1$ or $M_2$, say, $M_1 \in \C$, then $\C_1^M \subset \C$,
which yields that $\C_2^M = g_{21}(\C_1^M)$ is also a unit circular arc,
implying that $\C_1^M \cup \C_2^M$ belongs to the same unit circle $\hat \SS$.
Since this circle is invariant under a rotation around $O$,
we have $\hat \SS=\SS^1$, which contradicts our assumption
that $M,M_1,M_2 \in \inter \BB^2$.
Assume that $M_1, M_2 \notin \C$, and let $\C^1$ and $\C^2$ denote
$\C \cap \C_1^M$ and $\C \cap \C_2^M$, respectively.
Then $g_{21}(\C^2)$ is a unit circular arc in $\C_1^M$
whose length is equal to that of $\C^2$.
Thus, $g_{21}(\C^2)$ and $\C^1$ intersect in a unit circular arc,
which yields that $g_{21}(\C^2) \cup \C^1 = \C_1^M$ is a unit circular arc,
which leads to a contradiction in a similar way.

Let us say that a unit circular arc in $\bd \D_i$ is \emph{maximal},
if it is not a proper subset of another unit circular arc in $\bd \D_i$.
By Lemma~\ref{lem:Jordan}, $\bd \D_1$ contains finitely many,
say $m \geq 1$ maximal unit circular arcs of length $\alpha$,
one of which is $\SS_1$.
Thus, $\bd \D_3$ also contains $m$ maximal unit circular arcs of length $\alpha$.
By the previous paragraph, any of these arcs is contained in $\C_1^S \cup \C_1^M$
or in $\C_2^S \cup \C_2^M$.
Assume that all these arcs are contained in $\C_1^S$ or in $\C_2^S$.
Since $\C_1^S$, $\C_2^S$ and $\C_3^S$ are congruent,
we have that the total number of unit circular arcs of length $\alpha$ in $\C_i^S$
is equal to ${m}/{2}$. Thus, $\bd \D_1$ contains $m+1$ arcs,
which is a contradiction.

Finally, consider the case that some maximal unit circular arc $\SS_{\alpha}$
of length $\alpha$ in $\bd \D_3$ is not contained in $\C_1^S \cup \C_2^S$.
Since $\alpha > \alpha - \beta$, $M$ is not an interior point of $\SS_{\alpha}$,
but $M_1$ or $M_2$ is. Without loss of generality,
we may assume that $M_1$ is in the interior of $\SS_{\alpha}$.
This implies that $M$ is in the interior of
$g_{12}(\SS_{\alpha}) \subseteq \C_3^S \cup \C_2^M$
(similarly as Figure~\ref{fig:threeparts} shows).
Hence, $M$ is not an interior point of a unit circular arc in $\bd \D_1$,
which implies that $M_2$ is not an interior point of any unit circular arc
in $\bd \D_2$.
On the other hand, again by Lemma~\ref{lem:Jordan},
$\bd \D_3$ contains $k$ maximal unit circular arcs of length $\beta$
for some $k \geq 1$, one of which is $\SS_3$.
By our previous argument,
any of these arcs is contained in one of $\C_i^M$ or $\C_i^S$ for some $i \in \{ 1,2 \}$.
Let $k_M \geq 0$ and $k_S \geq 0$ denote the number of these arcs in $\C_1^M$ and $\C_1^S$, respectively.
Then $\C_1^M$ and $\C_1^S$ contain exactly $k_M$ and $k_S$ of these arcs, respectively.
From this it readily follows that $k=2k_M+2k_S+1$.
Furthermore, since $\bd \D_1$ also contains $k$ maximal unit circular arcs of length $\beta$,
we have $k=k_M+2k_S$. This yields that $k_M=-1$, which is a contradiction.

\goodbreak%
\noindent
\emph{\textbf{Case 3}:
 all of the arcs $\SS_1$, $g_{21}(\SS_2)$ and $g_{31}(\SS_3)$ coincide.}%
\\
In this case $g_{21}$ and $g_{31}$ are either reflections
about a line through $O$, or rotations around $O$.
In particular,
$O$ is a fixed point of both of them %$g_{21}$
and thus it is the unique common point $M$ of all tiles.
For any $i \neq j$, let $\C_{ij} = \D_i \cap \D_j$.
If both $g_{12}$ and $g_{13}$ are rotations around $O$,
then the tiling is clearly rotationally generated.
Hence, assume that one of $g_{12}$ and $g_{13}$,
say $g_{12}$ is a reflection about a line $\ell$ through $O$.
Then $g_{12}(\C_{13} \cup \C_{12}) = \C_{12} \cup \C_{23}$
yields that $\C_{12}$ is a straight line segment in $\ell$,
which, by the congruence of the tiles implies also that $C_{ij}$
is a segment for all $i \neq j$.
Thus, also in this case the tiling is rotationally generated,
and the assertion follows.

\null\vskip-2.5\baselineskip\null

\section{Remarks and open problems}\label{sec:remarks}

%First we mention that the picture on the second panel of Figure~\ref{fig:twelve}
%has been chosen as the logo of the MASS Program at Penn State \cite{MASS},
%and it was also featured on the front page of five issues of the Hungarian
%problem-solving journal K\"oz\'episkolai Matematikai Lapok \cite{KoMaL}.

First, we observe that the quantity $n(\K)$ can be similarly defined
for any $O$-symmetric convex body $\K$ in $\Re^d$ playing the role of $\BB^2$.
On the other hand, Theorem~\ref{thm:main} cannot be generalized
for any $O$-symmetric convex body even in the case $d=2$.
Indeed, taking a parallelogram and dissecting it into three congruent parallelograms
with two lines parallel to a pair of sides of the parallelogram shows
that there are $O$-symmetric plane convex bodies $\K$ with $n(\K)=3$.
However, it is easy to see that the following generalization
of Theorem~\ref{thm:main} holds.

\begin{thm}\label{thm:general}
If there is a monohedral tiling of an $O$-symmetric, strictly convex, smooth body $\K$ in $\Re^2$ with $k \leq 3$ topological discs,
then both $K$ and its tiling has a $k$-fold symmetry. In particular, for any $O$-symmetric, strictly convex plane body $\K$ of smooth boundary
 we have $n(\K) \geq 4$.
\end{thm}

This raises the question what happens if smoothness or the strictness of the convexity
is dropped from the conditions.

Following \cite{Goncharov}, we generalize Question~\ref{ques:existence}
for balls in arbitrary dimensions.
%Goncharov \cite{Goncharov} modified Question~\ref{ques:existence}
%to monohedral \emph{coverings} of $O$-symmetric convex bodies in $\Re^d$
%by subsets of the body.
%He proved that the cardinality of a nontrivial monohedral covering
%of an $O$-symmetric convex body is at least $d+1$.
%So the question remains open in dimensions $d \geq 3$:

\begin{ques}
 Are there monohedral tilings of the closed unit ball $\BB^d$
 such that the center of the ball is not contained in all of the tiles?
 More specifically, what are the values of $d$ for which it is possible?
\end{ques}

%Further, based on the proof of Theorem~\ref{thm:main}
We also raise the following, related problem:

\begin{ques}
 If $\BB^2$ has a tiling with \emph{similar} copies of some topological disc $\D$,
 does it follow that the tiles are congruent?
 Does it follow that $\BB^2$ has a tiling with \emph{congruent} copies of $\D$?
 Do these properties hold under some additional assumption on the tiles,
 e.g. if they have piecewise analytic boundaries?
\end{ques}

We should finally mention the \emph{divisibility problem},
in which the topological conditions on the tiles are dropped:
A subset of $\mathbb R^d$ is \emph{$m$-divisible}
if it can be decomposed into $m\in\mathbb N$ mutually \emph{disjoint} congruent subsets.
It is proved that typical convex bodies are not divisible \cite{Richter},
but balls are not typical in this sense, and they are $m$-divisible for large values of $m$
if $d$ is divisible by three \cite{KissLaczkovich}
or $d$ is even \cite{KissSomlai}.

\vspace{-.5\baselineskip}

\end{document}